\documentclass[pdflatex,sn-mathphys-num]{sn-jnl}
\usepackage{graphicx}%
\usepackage{multirow}%
\usepackage{amsmath,amssymb,amsfonts}%
\usepackage{amsthm}%
\usepackage{mathrsfs}%
\usepackage[title]{appendix}%
\usepackage{xcolor}%
\usepackage{textcomp}%
\usepackage{manyfoot}%
\usepackage{booktabs}%
\usepackage{algorithm}%
\usepackage{algorithmicx}%
\usepackage{algpseudocode}%
\usepackage{listings}%

\theoremstyle{thmstyleone}%
\newtheorem{theorem}{Theorem}
\newtheorem{corollary}[theorem]{Corollary}%

\theoremstyle{thmstyletwo}%

\theoremstyle{thmstylethree}%

\begin{document}

\title[Article Title]{A statistical theory of graph regularization for RNA velocity near developmental bifurcations}

\author*[1]{\fnm{Lingqi} \sur{Meng}}\email{lingqime@ybu.edu.cn}

\author[2]{\fnm{Shiruo} \sur{Wang}}\email{shiruomath@gmail.com}

\affil*[1]{\orgdiv{Department of Mathematics, Yanbian University, Yanbian Korean Autonomous Prefecture, Jilin 133002, China}}

\affil[2]{\orgname{Department of Mathematics, State University of New York at Buffalo, Buffalo, New York 14260, United States}}


\abstract{Graph-based regularization is widely used to stabilize noisy RNA-velocity estimates by encouraging transcriptionally similar cells to share similar velocity vectors. Near developmental bifurcations, however, proximity-based graphs may connect cells from distinct daughter lineages, reducing estimation variance at the cost of attenuating biologically meaningful lineage-specific dynamics. We formulate graph-regularized RNA velocity as a statistical estimation problem on a potentially misspecified cell-state graph and develop a theoretical framework for analyzing this trade-off. We derive an exact graph-spectral bias--variance decomposition that characterizes how Laplacian regularization suppresses estimation noise while introducing systematic smoothing bias. To quantify lineage preservation, we introduce a branch-sensitive risk that separates within-lineage denoising from cross-lineage information leakage. We further show that persistent cross-branch connectivity can induce nonvanishing branch bias, implying that standard proximity-graph regularization may remain asymptotically inconsistent near developmental bifurcations. These results provide a mathematical foundation for understanding both the statistical benefits and the geometric limitations of graph regularization for RNA-velocity estimation.}


\maketitle

\section{Introduction}

Single-cell RNA sequencing provides static molecular measurements from large populations of cells, whereas biological processes such as differentiation and regeneration are inherently dynamic. RNA velocity was introduced to recover directional information from snapshot transcriptomic data using unspliced and spliced messenger RNA abundances \cite{la2018rna}, and was subsequently extended to transient cellular states through dynamical modeling \cite{bergen2020generalizing}.

RNA-velocity estimates, however, can be sensitive to modeling assumptions, sparse counts, preprocessing, and neighborhood-based smoothing \cite{bergen2021rna,gorin2022rna,zheng2023pumping}. In practice, local transcriptomic neighborhoods, commonly represented by $k$-nearest-neighbor or related similarity graphs, are used to stabilize noisy velocity-derived quantities \cite{la2018rna,bergen2020generalizing,lange2022cellrank}. Mathematically, such procedures are closely related to Laplacian regularization, where graph smoothness is quantified through a Dirichlet energy and regularization suppresses high-frequency graph components \cite{chung1997spectral,belkin2006manifold}.

The validity of this approach depends critically on whether the cell-state graph represents the geometry relevant to the dynamics. Proximity graphs are widely used to approximate low-dimensional structure from sampled data, and their graph Laplacians can converge to continuum differential operators under suitable geometric, sampling, and bandwidth assumptions \cite{hein2007graph, belkin2008towards}. In single-cell trajectory analysis, related neighborhood graphs are used to represent phenotypic manifolds and constrain transitions between transcriptionally similar cells \cite{haghverdi2016diffusion,setty2016wishbone,lange2022cellrank}.

Developmental bifurcations create a particular difficulty. Cells belonging to distinct daughter lineages may remain close in transcriptomic space while their future dynamics begin to diverge. A proximity graph may therefore contain cross-branch edges that are geometrically short but dynamically inappropriate. Smoothing across such edges reduces noise but can simultaneously attenuate lineage-specific velocity differences. This raises the central mathematical question of this work: under what conditions does graph regularization improve RNA-velocity estimation, and when does graph misspecification instead produce systematic loss of branch-specific dynamics?

We formulate graph-regularized RNA velocity as a statistical estimation problem on a potentially misspecified cell-state graph. For Laplacian regularization, we derive an exact graph-spectral bias--variance decomposition showing how spectral filtering trades variance reduction against distortion of the true velocity field. We then introduce a branch-sensitive risk and an exactly solvable two-branch model, from which we obtain explicit formulas for attenuation of daughter-lineage velocity contrasts by cross-branch connectivity.

Our main asymptotic result identifies a distinct failure mechanism from classical graph-Laplacian consistency theory. Near a developmental bifurcation, proximity-based graph construction can retain a nonvanishing amount of cross-branch connectivity as the sample size increases. When the resulting effective cross-branch conductance persists at the regularization scale, the branch-contrast error remains bounded away from zero. Thus, increasing sample size alone does not guarantee branch preservation: consistency requires the effective strength of cross-branch smoothing to vanish.

The framework is deliberately conditional on an initial RNA-velocity estimator. Rather than modifying the kinetic model used to infer velocity, we isolate the subsequent graph-regularization problem and characterize how its statistical behavior depends on developmental geometry. This separates errors arising from kinetic inference from those induced by graph smoothing and identifies graph misspecification as a fundamental source of bias near developmental bifurcations.

\section{Mathematical models}
\label{sec:framework}

This section introduces the mathematical framework used throughout the paper. We distinguish the underlying biological dynamics, the sampled cell states, the initial RNA-velocity estimator, and the graph used for regularization.

Before introducing the mathematical framework, we briefly illustrate the role
of graph regularization within a typical RNA velocity analysis pipeline (Fig.~\ref{fig:pipeline}). Starting from an initial velocity estimate, a cell-state graph is constructed
from transcriptionally similar cells, after which graph regularization
propagates velocity information across neighboring cells to produce a smoother
velocity field. The present work focuses exclusively on this graph
regularization step and analyzes its statistical behavior under graph
misspecification.

\begin{figure}[t]
    \centering
    \includegraphics[width=\textwidth]{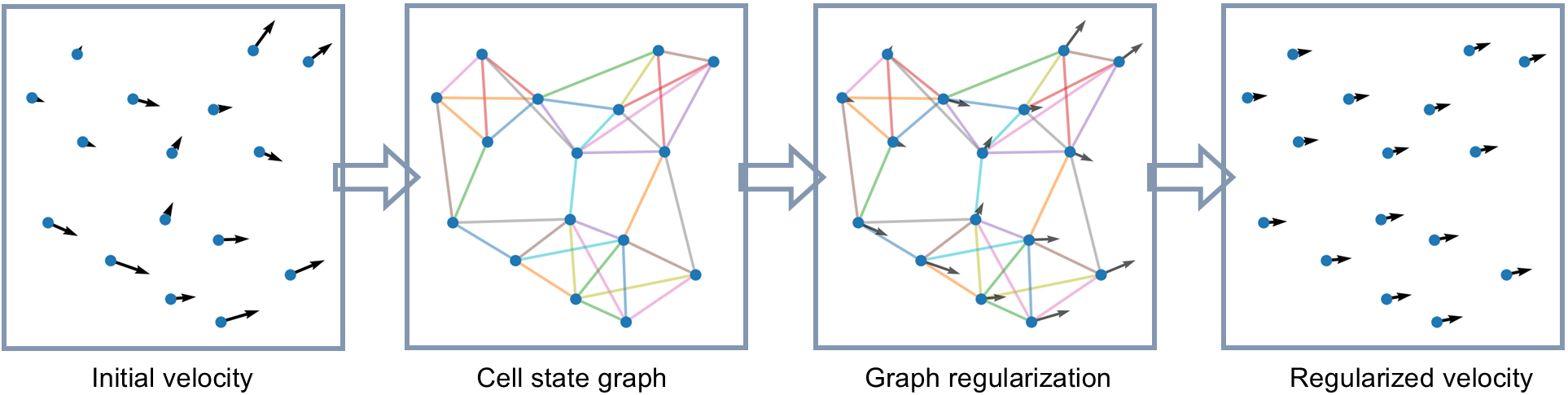}
    \caption{
    Illustration of the graph-regularized RNA velocity pipeline.
    An initial velocity estimate is represented as a vector field on the sampled
    cell states. A cell-state graph is then constructed from transcriptional
    similarity, and graph regularization smooths the velocity field over the graph
    to obtain the final regularized estimate. 
    }
    \label{fig:pipeline}
\end{figure}

Throughout the paper, vectors are column vectors. For a matrix $A$, we write $A^\top$ for its transpose, $\|A\|_F$ for its Frobenius norm, $\|A\|_2$ for its spectral norm, and $\operatorname{tr}(A)$ for its trace. The identity matrix is denoted by $I_N$, or simply by $I$ when its dimension is clear. For a symmetric matrix $A$, the notation $A\succeq0$ means that $A$ is positive semidefinite.

\subsection{Cell-state dynamics and sampled velocity field}
\label{subsec:cell-dynamics-sampled-field}

Let $\mathcal{M}$ denote the latent cell-state space. We assume that $\mathcal{M}$ is a compact subset of $\mathbb{R}^p$, or more generally a finite union of smooth low-dimensional branches embedded in $\mathbb{R}^p$. The ambient dimension $p$ may represent gene-expression features, principal components, or another representation used for graph construction.

A biological trajectory is modeled by a differentiable curve
\[
x:[0,1]\longrightarrow\mathcal{M},
\]
with property
\begin{equation}
\frac{\mathrm{d}x(t)}{\mathrm{d}t} = v^\star\bigl(x(t)\bigr),
\label{eq:continuous-biological-ode}
\end{equation}
where
\[
v^\star:\mathcal{M}\longrightarrow\mathbb{R}^q
\]
is the unknown biological velocity field. The dimensions $p$ and $q$ need not coincide; for example, the graph may be constructed in a low-dimensional embedding while velocities are represented in another coordinate system.

The observed cells are regarded as independent samples from a population whose local dynamics are governed by the common vector field $v^\star$. To be precise, let $X_1,\ldots,X_N\in\mathcal{M}$ be the observed cell states, assumed to be independent samples from a distribution $P_X$ supported on $\mathcal{M}$. For each cell, define $V_i^\star = v^\star(X_i) \in\mathbb{R}^q$, and collect these vectors into
\begin{equation}
V^\star =
\begin{bmatrix}
    (V_1^\star)^\top\\
    \vdots\\
    (V_N^\star)^\top
\end{bmatrix}
\in\mathbb{R}^{N\times q},
\label{eq:true-velocity-matrix}
\end{equation}
whose $i$-th row is the true velocity of cell $i$. Throughout the paper, $v^\star$ denotes the continuous biological velocity field, whereas $V^\star$ denotes its evaluation on the observed sample.

\subsection{Initial RNA-velocity estimator}
\label{subsec:initial-estimator}

Let $\widehat{V}^{(0)} \in \mathbb{R}^{N\times q}$ denote an initial RNA-velocity estimate obtained from an arbitrary kinetic estimation procedure. The proposed framework is independent of the underlying kinetic model and analyzes only the statistical properties of the subsequent graph-regularization step.

We write
\begin{equation}
\widehat{V}^{(0)} = V^\star + B + \Xi,
\label{eq:general-observation-model}
\end{equation}
where
\[
B = \mathbb{E} \!\left[
\widehat{V}^{(0)} - V^\star \,\middle|\, X_1,\ldots,X_N
\right]
\]
is the conditional bias and
\[
\mathbb{E} \!\left[ 
\Xi \,\middle|\, X_1,\ldots,X_N
\right]
= 0
\]
is the centered stochastic error.

For most theoretical results in this paper, we consider the conditionally unbiased model
\begin{equation}
\widehat{V}^{(0)} = V^\star + \Xi,
\label{eq:unbiased-observation-model}
\end{equation}
since the general model is recovered by replacing $V^\star$ with $V^\star+B$.

When explicit expressions are required (see Theorem \ref{thm:exact-spectral-risk}), we further assume the isotropic conditional covariance model
\begin{equation}
\operatorname{Cov}
\!\left(
\operatorname{vec}(\Xi)
\,\middle|\,
X_1,\ldots,X_N
\right)
=
\sigma^2I_{Nq},
\label{eq:isotropic-noise}
\end{equation}
where $\operatorname{vec}(\Xi)$ represents the vectorization of matrix $\Xi$.

For notational simplicity, we omit the conditioning notation and write
\[
\mathbb{E}[\cdot]
\qquad\text{and}\qquad
\operatorname{Cov}(\cdot)
\]
instead of
\[
\mathbb{E}
\!\left[
\cdot
\,\middle|\,
X_1,\ldots,X_N
\right]
\qquad\text{and}\qquad
\operatorname{Cov}
\!\left(
\cdot
\,\middle|\,
X_1,\ldots,X_N
\right),
\]
respectively, whenever no ambiguity arises.

\subsection{Cell-state graph}

Let $\mathcal{G} = (\mathcal{V},\mathcal{E},W)$ be an undirected weighted graph on the observed cells, where $\mathcal{V}=\{1,\ldots,N\}$ and $W=(w_{ij})$ is symmetric with $w_{ij}\ge0$ and $w_{ii}=0$. Let
\[
D=\operatorname{diag}(d_1,\ldots,d_N),
\qquad
d_i=\sum_{j=1}^Nw_{ij},
\]
and define the graph Laplacian $L=D-W$. One can verify
\[
L=L^\top,
\qquad
L\succeq0,
\qquad
L\mathbf1_N=0.
\]

For any $V\in\mathbb{R}^{N\times q}$, its graph Dirichlet energy is
\begin{equation}
\mathcal{E}_L(V) = \operatorname{tr}(V^\top LV) 
= \frac{1}{2} \sum_{i,j} w_{ij} \|V_i-V_j\|_2^2.
\label{eq:dirichlet-edge-form}
\end{equation}
This quadratic form penalizes velocity differences across adjacent cells and serves as the smoothness regularizer throughout the paper.

Next, we consider proximity graph construction. Let $Z_i\in\mathbb{R}^p$ denote the representation used for graph construction. A generic proximity graph is defined by
\begin{equation}
w_{ij}^{\mathrm{prox}} = K\!\left( \frac{\|Z_i-Z_j\|_2}{h_N} \right)
\mathbf{1} \left\{i\in\mathcal N_k(j)
\text{ or }
j\in\mathcal N_k(i)\right\},
\label{eq:proximity-weight}
\end{equation}
where $K$ is a nonincreasing kernel, $h_N$ is the bandwidth, and $\mathcal N_k(i)$ denotes the $k$ nearest neighbors of cell $i$. The resulting graph is denoted by $\mathcal G_{\mathrm{obs}} = (\mathcal V,\mathcal E_{\mathrm{obs}},W_{\mathrm{obs}})$ with Laplacian $L_{\mathrm{obs}} = D_{\mathrm{obs}}-W_{\mathrm{obs}}$.

Near developmental bifurcations, $L_{\mathrm{obs}}$ may contain edges connecting cells from distinct daughter lineages despite their different future dynamics.

\subsection{Branch structure and the oracle branch graph}
\label{subsec:branch-structure}

Suppose
\[
\mathcal M
=
\bigcup_{b=1}^K
\mathcal M_b,
\]
where the branches intersect only at bifurcation points.

Let
\[
g_i\in\{1,\ldots,K\}
\]
denote the branch label of cell $i$.

Define
\begin{equation}
(W_{\mathrm{in}})_{ij}
=
w_{ij}\mathbf1\{g_i=g_j\},
\end{equation}

\begin{equation}
(W_{\mathrm{cross}})_{ij}
=
w_{ij}\mathbf1\{g_i\neq g_j\},
\end{equation}

so that

\[
W_{\mathrm{obs}}
=
W_{\mathrm{in}}
+
W_{\mathrm{cross}}.
\]

Their Laplacians satisfy

\begin{equation}
L_{\mathrm{obs}}
=
L_{\mathrm{in}}
+
L_{\mathrm{cross}},
\label{eq:laplacian-decomposition}
\end{equation}

where
$L_{\mathrm{in}}$
serves as an oracle graph describing ideal within-branch connectivity.

For any
$V$,

\begin{equation}
\operatorname{tr}
(V^\top L_{\mathrm{obs}}V)
=
\operatorname{tr}
(V^\top L_{\mathrm{in}}V)
+
\operatorname{tr}
(V^\top L_{\mathrm{cross}}V),
\label{eq:within-cross-energy}
\end{equation}

which separates within-branch smoothing from cross-branch smoothing.

\subsection{Graph-regularized velocity estimator}

Given an initial estimator $\widehat V^{(0)}$, a graph Laplacian $L$, and a regularization parameter $\rho\ge0$, define
\begin{equation}
\widehat V_\rho(L) = \arg\min_{V\in\mathbb R^{N\times q}}
\left\{ \| \widehat V^{(0)} - V \|_F^2 + \rho \operatorname{tr}(V^\top LV) \right\}.
\label{eq:graph-regularized-objective}
\end{equation}
The optimality condition gives
\begin{equation}
(I+\rho L) \widehat V_\rho(L) = \widehat V^{(0)},
\label{eq:regularized-linear-system}
\end{equation}
and therefore
\begin{equation}
\widehat V_\rho(L) = (I+\rho L)^{-1} \widehat V^{(0)}.
\label{eq:closed-form-estimator}
\end{equation}

Define
\begin{equation}
S_\rho(L) = (I+\rho L)^{-1},
\label{eq:smoothing-operator}
\end{equation}
so that
\[
\widehat V_\rho(L)
=
S_\rho(L)
\widehat V^{(0)}.
\]
Since $L$ is symmetric positive semidefinite, it admits the eigendecomposition $L = \Phi\Lambda\Phi^\top$, where
\[
\Lambda
=
\operatorname{diag}
(\lambda_1,\ldots,\lambda_N),
\qquad
0=\lambda_1\le\cdots\le\lambda_N.
\]
Thus, we have
\begin{equation}
S_\rho(L) = \Phi (I+\rho\Lambda)^{-1} \Phi^\top,
\label{eq:spectral-smoothing-operator}
\end{equation}
and
\begin{equation}
\widehat V_\rho(L)
=
\sum_{k=1}^N
\frac{1}{1+\rho\lambda_k}
\phi_k\phi_k^\top
\widehat V^{(0)}.
\label{eq:spectral-filter}
\end{equation}

Thus graph regularization acts as a spectral filter that increasingly suppresses higher graph-frequency components.

\subsection{Conditional expectation, bias, and covariance}
\label{subsec:conditional-moments}

Under the unbiased observation model
\eqref{eq:unbiased-observation-model}, and conditional on the observed cell
states and graph,
\begin{equation}
    \widehat{V}_{\rho}(L)
    -
    V^\star
    =
    \bigl(S_{\rho}(L)-I\bigr)V^\star
    +
    S_{\rho}(L)\Xi.
    \label{eq:error-decomposition}
\end{equation}

The conditional expectation is
\begin{equation}
    \mathbb{E}
    \left[
        \widehat{V}_{\rho}(L)
        \,\middle|\,
        X_1,\ldots,X_N,L
    \right]
    =
    S_{\rho}(L)V^\star.
    \label{eq:conditional-expectation-estimator}
\end{equation}
Hence the graph-induced conditional bias is
\begin{equation}
    \operatorname{Bias}_{\rho}(L)
    =
    \left(
        S_{\rho}(L)-I
    \right)V^\star
    =
    -\rho
    (I+\rho L)^{-1}
    LV^\star.
    \label{eq:graph-induced-bias}
\end{equation}

The second equality follows from the resolvent identity
\[
(I+\rho L)^{-1}-I
=
-\rho(I+\rho L)^{-1}L.
\]
It shows that regularization introduces no bias when $LV^\star=0$, that is,
when the true sampled velocity field is constant on every connected
component of the graph.

For each coordinate $r\in\{1,\ldots,q\}$, let
\[
\Xi_{\cdot r}
=
(\Xi_{1r},\ldots,\Xi_{Nr})^\top
\]
denote the noise vector for that coordinate, with covariance
\[
\Sigma_r
=
\operatorname{Cov}
\left(
    \Xi_{\cdot r}
    \,\middle|\,
    X_1,\ldots,X_N
\right).
\]
Then
\begin{equation}
    \operatorname{Cov}
    \left(
        \widehat{V}_{\rho,\cdot r}(L)
        \,\middle|\,
        X_1,\ldots,X_N,L
    \right)
    =
    S_{\rho}(L)
    \Sigma_r
    S_{\rho}(L).
    \label{eq:regularized-coordinate-covariance}
\end{equation}
Under isotropic noise,
\[
\Sigma_r
=
\sigma^2 I_N,
\]
and therefore
\begin{equation}
    \operatorname{Cov}
    \left(
        \widehat{V}_{\rho,\cdot r}(L)
    \right)
    =
    \sigma^2 S_{\rho}(L)^2.
    \label{eq:isotropic-regularized-covariance}
\end{equation}

These expressions provide the basis for the exact bias--variance theory
developed in the next section.

\section{Results}

\subsection{Graph-spectral risk theory}
\label{sec:graph-risk}

In this section, we characterize the statistical effect of quadratic graph regularization. Throughout subsections~\ref{sec:graph-risk}--\ref{sec:branch-leakage}, we condition on the observed cell states and on the graph Laplacian $L$. Accordingly, $V^\star$ and $L$ are treated as fixed, whereas the initial velocity noise $\Xi$ is random.

Recall that
\[
\widehat V^{(0)}=V^\star+\Xi, 
\qquad 
\widehat V_\rho = S_\rho\widehat V^{(0)},
\qquad
S_\rho=(I+\rho L)^{-1}.
\]
Let $L=\Phi \Lambda \Phi^\top, \Lambda=\operatorname{diag}(\lambda_1,\ldots,\lambda_N)$, where $\Phi=[\phi_1,\ldots,\phi_N]$ is orthogonal and $0\leq\lambda_1\leq\cdots\leq\lambda_N$.
We write
\[
\widetilde V^\star=\Phi^\top V^\star
\]
for the graph Fourier representation of the true velocity field, and $\widetilde V^\star_k\in\mathbb R^q$ for its $k$-th row.
\subsubsection{Exact bias--variance decomposition}
\label{subsec:exact-risk}

The regularized estimation error admits the deterministic--stochastic decomposition. The first term is the graph-induced bias and the second term is the filtered initial estimation noise.

\begin{theorem}
\label{thm:exact-spectral-risk}
Assume $\mathbb E[\Xi]=0$ and $\operatorname{Cov}\!\left(\operatorname{vec}(\Xi)\right) = \sigma^2 I_{Nq}$. Then, we have the graph-spectral risk decomposition
\begin{align}
\mathcal R_{\mathrm{global}}(\rho;L,V^\star) 
&:= \mathbb E \left[ \|\widehat V_\rho-V^\star\|_F^2 \right] \nonumber \\
&\ =\sum_{k=1}^N \left( \frac{\rho\lambda_k}{1+\rho\lambda_k} \right)^2 \|\widetilde V^\star_k\|_2^2 + 
q\sigma^2 \sum_{k=1}^N \frac{1}{(1+\rho\lambda_k)^2}.
\label{eq:exact-spectral-risk}
\end{align}
The first term in Eq.~\eqref{eq:exact-spectral-risk} is the squared spectral attenuation bias, and the second term is the filtered noise variance.
\end{theorem}

\begin{proof}
See Appendix~\ref{secA1}.
\end{proof}

The exact risk decomposition in
Theorem~\ref{thm:exact-spectral-risk} shows that graph regularization
acts through two opposing mechanisms: it suppresses stochastic variation
while attenuating nonconstant components of the true velocity field.
The following proposition summarizes the corresponding spectral
properties.

\begin{corollary}[Spectral contraction and monotone bias--variance behavior]
\label{prop:spectral-bias-variance-properties}
Let
\[
S_\rho=(I+\rho L)^{-1},
\qquad
L=\Phi\Lambda\Phi^\top,
\qquad
\Lambda=\operatorname{diag}(\lambda_1,\ldots,\lambda_N),
\]
where \(L\) is symmetric positive semidefinite and \(\rho\geq0\).
Then the following statements hold.

\begin{enumerate}
    \item The graph smoother is nonexpansive:
    \[
    \|S_\rho\|_2
    =
    \max_{1\leq k\leq N}
    \frac{1}{1+\rho\lambda_k}
    \leq 1.
    \]
    If \(\rho>0\), then \(S_\rho\) is a strict contraction on
    \(\ker(L)^\perp\).

    \item Under the isotropic noise assumption of
    Theorem~\ref{thm:exact-spectral-risk}, the variance term
    \[
    \mathcal V(\rho)
    =
    q\sigma^2\operatorname{tr}(S_\rho^2)
    =
    q\sigma^2
    \sum_{k=1}^N
    \frac{1}{(1+\rho\lambda_k)^2}
    \]
    is nonincreasing, with
    \begin{equation}
    \mathcal V'(\rho)
    =
    -2q\sigma^2
    \sum_{k=1}^N
    \frac{\lambda_k}{(1+\rho\lambda_k)^3}
    \leq 0.
    \label{eq:variance-derivative}
    \end{equation}
    If \(L\neq0\) and \(\sigma^2>0\), then
    \(\mathcal V'(\rho)<0\) for every \(\rho\geq0\).

    \item The squared-bias term
    \[
    \mathcal B(\rho)
    =
    \sum_{k=1}^N
    \left(
        \frac{\rho\lambda_k}{1+\rho\lambda_k}
    \right)^2
    \|\widetilde V_k^\star\|_2^2
    \]
    is nondecreasing, with
    \begin{equation}
    \mathcal B'(\rho)
    =
    2
    \sum_{k=1}^N
    \frac{\rho\lambda_k^2}
    {(1+\rho\lambda_k)^3}
    \|\widetilde V_k^\star\|_2^2
    \geq0.
    \label{eq:bias-derivative}
    \end{equation}
    For \(\rho>0\), the inequality is strict whenever
    \(LV^\star\neq0\).
\end{enumerate}
\end{corollary}

\begin{proof}
The eigenvalues of \(S_\rho\) are $\frac{1}{1+\rho\lambda_k}$, where $k=1,\ldots,N$. They belong to \((0,1]\), which proves the nonexpansiveness statement. On \(\ker(L)^\perp\), all relevant eigenvalues of \(L\) are positive, so the corresponding eigenvalues of \(S_\rho\) are strictly smaller than one whenever \(\rho>0\).

The expressions for \(\mathcal V'(\rho)\) and \(\mathcal B'(\rho)\) follow by differentiating their respective spectral representations term by term. Strict decrease of the variance holds when at least one eigenvalue of \(L\) is positive. Strict increase of the squared bias holds when \(V^\star\) has a nonzero component in an eigenspace associated with a positive eigenvalue, which is equivalent to \(LV^\star\neq0\).
\end{proof}

Corollary~\ref{prop:spectral-bias-variance-properties} formalizes the
basic bias--variance trade-off. Increasing \(\rho\) suppresses the
nonconstant graph-frequency components of the estimation noise, but
the same spectral attenuation also removes nonconstant components of
the true velocity field. Nevertheless, sufficiently weak
regularization is always beneficial under nontrivial isotropic noise.

The graph-induced bias can also be controlled by the graph smoothness
of the true velocity field.

\begin{corollary}[Dirichlet-energy control of the regularization bias]
\label{prop:dirichlet-bias-bound}
For every \(\rho\geq0\),
\begin{equation}
\|(S_\rho-I)V^\star\|_F^2
\leq
\rho^2\|LV^\star\|_F^2
\leq
\rho^2\lambda_{\max}(L)
\operatorname{tr}
\left[
    (V^\star)^\top LV^\star
\right].
\label{eq:dirichlet-bias-bound}
\end{equation}
\end{corollary}

\begin{proof}
Using
\[
S_\rho-I=-\rho S_\rho L
\]
and Corollary~\ref{prop:spectral-bias-variance-properties},
\[
\|(S_\rho-I)V^\star\|_F
\leq
\rho\|S_\rho\|_2\|LV^\star\|_F
\leq
\rho\|LV^\star\|_F.
\]
Squaring gives the first inequality. For the second, since
\[
L^2\preceq\lambda_{\max}(L)L,
\]
we have
\[
\|LV^\star\|_F^2
=
\operatorname{tr}
\left[
    (V^\star)^\top L^2V^\star
\right]
\leq
\lambda_{\max}(L)
\operatorname{tr}
\left[
    (V^\star)^\top LV^\star
\right].
\]
\end{proof}

The bound shows that regularization introduces little bias when the true velocity field has small graph Dirichlet energy. Conversely, a large value of $\operatorname{tr} \left[(V^\star)^\top LV^\star \right]$ indicates that the graph connects cells with substantially different true velocities, in which case graph smoothing may strongly distort the underlying biological signal.

\subsubsection{Numerical validation of the bias--variance decomposition}
\label{subsec:numerical-bias-variance}

We first performed a controlled synthetic experiment to validate the exact
bias--variance decomposition derived in
Theorem~\ref{thm:exact-spectral-risk}. A synthetic data set was used rather
than a real single-cell data set because the true velocity field
\(V^\star\) must be known in order to separately evaluate the graph-induced
bias and the filtered noise variance.

We generated \(N=200\) cells along a smooth one-dimensional trajectory embedded in \(\mathbb{R}^2\). Specifically, for $t_i=\frac{i-1}{N-1}$, $i=1,\ldots,N$,  the cell-state coordinates were defined by
\begin{equation}
X_i = \gamma(t_i) =
\begin{pmatrix}
t_i\\
0.35\sin(2\pi t_i)
\end{pmatrix}.
\label{eq:synthetic-cell-trajectory}
\end{equation}
This construction produces a smooth, nonbranching developmental trajectory, shown in Fig.~\ref{fig:bias-variance-validation}(a). The graph was constructed from these cell-state coordinates using a symmetric weighted \(k\)-nearest-neighbor graph with \(k=10\). For neighboring cells, the graph weights were defined using a Gaussian kernel,
\begin{equation}
w_{ij} = \exp \left( -\frac{\|X_i-X_j\|_2^2}{2h^2}\right),
\label{eq:synthetic-graph-weight}
\end{equation}
where \(h\) was chosen as the median distance among the retained nearest-neighbor edges. The corresponding unnormalized graph Laplacian was \(L=D-W\).

To obtain a geometrically interpretable velocity field, we defined the true velocity at each cell as the tangent vector of the underlying trajectory, i.e.,
\begin{equation}
V_i^\star = \gamma'(t_i) =
\begin{pmatrix}
1\\
0.7\pi\cos(2\pi t_i)
\end{pmatrix}.
\label{eq:synthetic-true-velocity}
\end{equation}
Thus, the true velocity vectors are aligned with the direction of progression
along the synthetic trajectory. The initial velocity estimate was generated
according to
\begin{equation}
\widehat V^{(0)} = V^\star+\Xi,
\qquad
\Xi_{ir} \overset{\mathrm{iid}}{\sim} \mathcal{N}(0,\sigma^2),
\label{eq:synthetic-noisy-velocity}
\end{equation}
with noise standard deviation \(\sigma=0.6\). A representative realization of the noisy initial velocity field is also displayed in Fig.~\ref{fig:bias-variance-validation}(a). Whereas the true velocity vectors follow the local tangent direction, the noisy vectors exhibit substantial cell-to-cell fluctuations in both magnitude and orientation.

For each regularization parameter $\rho\in[10^{-4},10^3]$, we computed the graph-regularized estimator $\widehat V_\rho = (I+\rho L)^{-1}\widehat V^{(0)}$. The theoretical squared bias was evaluated as $\mathcal{B}(\rho) = \frac{1}{Nq} \left\| (S_\rho-I)V^\star \right\|_F^2$ with $S_\rho=(I+\rho L)^{-1}$. The theoretical variance was evaluated as $\mathcal{V}(\rho) = \frac{\sigma^2}{N} \operatorname{tr}(S_\rho^2)$, where \(q=2\) is the velocity dimension. Equivalently, using the eigendecomposition
\(L=\Phi\Lambda\Phi^\top\), these quantities were computed from
\begin{align}
\mathcal{B}(\rho) &= \frac{1}{Nq} \sum_{k=1}^{N} \left( \frac{\rho\lambda_k}          {1+\rho\lambda_k} \right)^2 \left\| \widetilde V_k^\star \right\|_2^2, 
\label{eq:synthetic-spectral-bias} \\ 
\mathcal{V}(\rho) &= \frac{\sigma^2}{N} \sum_{k=1}^{N} \frac{1}{(1+\rho\lambda_k)^2}.
\label{eq:synthetic-spectral-variance}
\end{align}
Their sum, $\mathcal{R}(\rho)=\mathcal{B}(\rho)+\mathcal{V}(\rho)$, is the theoretical mean squared error per velocity coordinate.

We additionally estimated the risk by Monte Carlo simulation. The cell locations, graph Laplacian, and true velocity field were held fixed, while \(M=500\) independent noise matrices $\Xi^{(1)},\ldots,\Xi^{(M)}$ were generated according to Eq.~\eqref{eq:synthetic-noisy-velocity}. For each realization and each value of \(\rho\), we computed $\widehat V_\rho^{(m)} = S_\rho \left( V^\star+\Xi^{(m)} \right)$, and approximated the expected risk by
\begin{equation}
\widehat{\mathcal{R}}_{\mathrm{MC}}(\rho) 
= \frac{1}{M} \sum_{m=1}^{M} \frac{1}{Nq} \left\| \widehat V_\rho^{(m)}-V^\star \right\|_F^2.
\label{eq:monte-carlo-risk}
\end{equation}
The Monte Carlo values are represented by open circles in Fig.~\ref{fig:bias-variance-validation}(b).

\begin{figure}[t]
\centering
\includegraphics[width=0.95\textwidth]{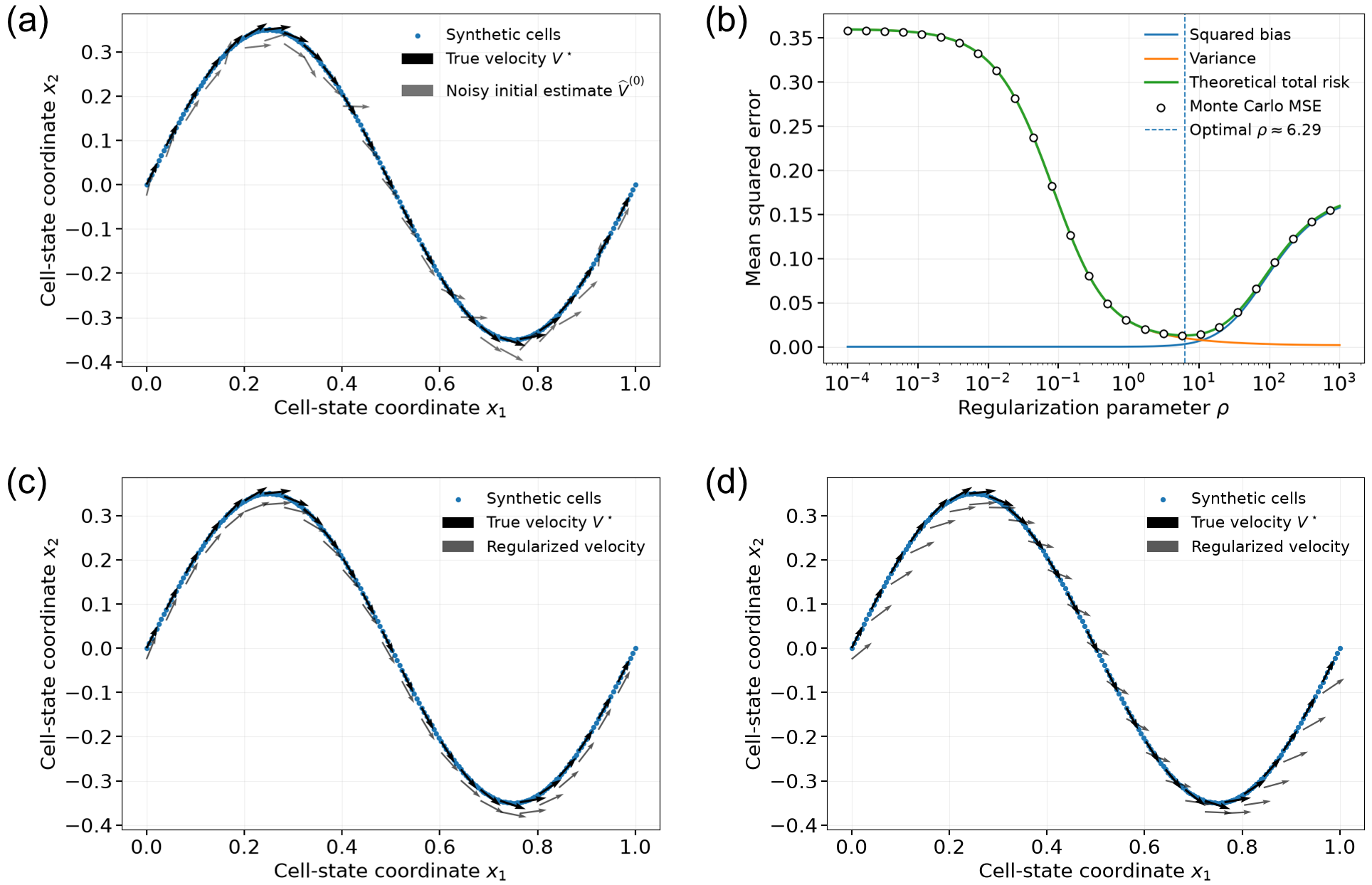}
\caption{
Numerical validation of the exact graph-spectral bias--variance decomposition.
(\textbf{a}) Synthetic cell-state trajectory, weighted $10$-nearest-neighbor graph,
true velocity field, and one realization of the noisy initial velocity estimate.
(\textbf{b}) Theoretical squared bias, variance, total risk, and empirical Monte
Carlo mean squared error computed from $500$ independent noise realizations.
(\textbf{c}) Graph-regularized velocity field for the near-optimal regularization
parameter $\rho=6.29$.
(\textbf{d}) Graph-regularized velocity field for strong regularization
($\rho=100$).
}
\label{fig:bias-variance-validation}
\end{figure}

Figure~\ref{fig:bias-variance-validation}(b) shows the expected
bias--variance trade-off. At \(\rho=0\), the estimator is unbiased, and the
total risk is entirely determined by the noise variance. As \(\rho\)
increases, the variance decreases monotonically because the graph smoother
attenuates an increasing number of nonconstant graph-frequency components.
In contrast, the squared bias increases because the same attenuation also
removes components of the true velocity field.

For small positive values of \(\rho\), the decrease in variance dominates the
increase in squared bias, and the total risk falls below the unregularized
risk. This observation is consistent with the local risk result that the
variance decreases to first order at \(\rho=0\), whereas the squared bias
increases only to second order. The total risk therefore attains an interior
minimum at a positive regularization parameter. Beyond this minimum, further
smoothing produces progressively larger distortion of the true velocity
field, and the increase in bias eventually outweighs the remaining variance
reduction.

The Monte Carlo estimates closely follow the exact theoretical risk curve
over the full range of regularization parameters. This agreement numerically
confirms the graph-spectral risk decomposition and shows that the observed
U-shaped risk profile is explained by the opposing effects of noise
suppression and signal attenuation. The experiment also demonstrates that
graph regularization is beneficial only within an intermediate range of
regularization strengths: insufficient smoothing leaves substantial
estimation noise, whereas excessive smoothing suppresses genuine variation
in the velocity field.

To further illustrate the effect of the regularization parameter on the
estimated velocity field, we compare the graph-regularized velocities obtained from the same noisy initial estimate
under two different values of \(\rho\).

When the regularization parameter is chosen near the theoretical optimum
(\(\rho=6.29\); Fig.~\ref{fig:bias-variance-validation}(c)), most of the
high-frequency fluctuations introduced by the observation noise are removed,
while the overall spatial pattern of the true velocity field is well
preserved. The regularized vectors remain closely aligned with the local
tangent direction of the underlying trajectory, indicating that graph
smoothing primarily suppresses stochastic variation without introducing
substantial distortion.

In contrast, when the regularization parameter is increased to
\(\rho=100\) (Fig.~\ref{fig:bias-variance-validation}(d)), the estimated velocity
field becomes noticeably over-smoothed. Although the local variability is
further reduced, the velocity vectors no longer accurately follow the
underlying spatial variation along the trajectory. Instead, neighboring
velocities become excessively similar, demonstrating the increasing influence
of graph-induced bias predicted by the theoretical analysis.

These observations provide a direct visual interpretation of the
bias--variance trade-off shown in
Fig.~\ref{fig:bias-variance-validation}(b). Moderate graph regularization
improves estimation accuracy by reducing variance while introducing only a
small bias, whereas excessive regularization eventually suppresses genuine
biological variation and increases the overall estimation error.

\subsection{Branch-contrast attenuation in two-block graphs}
\label{sec:branch-leakage}

\subsubsection{Exact attenuation and persistent asymptotic bias}

The global risk decomposition does not distinguish whether the graph-induced bias originates from smoothing within a lineage or from mixing distinct developmental lineages. To isolate the latter mechanism, we consider an exactly solvable two-branch model.

Let the vertex set be partitioned into two nonempty daughter branches,
\[
\mathcal V = \mathcal I_1 \cup \mathcal I_2, \qquad \mathcal I_1 \cap \mathcal I_2 = \varnothing, 
\] 
with $|\mathcal I_1| = n_1$ and $|\mathcal I_2| = n_2.$

Assume that the true velocity is constant within each branch, i.e.,
\begin{equation}
V_i^\star =
\begin{cases}
\mu_1, & i\in\mathcal I_1,\\
\mu_2, & i\in\mathcal I_2,
\end{cases}
\qquad \mu_1,\mu_2\in\mathbb R^q,
\label{eq:block-constant-velocity}
\end{equation}
and define the true branch contrast by $\Delta^\star = \mu_1-\mu_2$.

Suppose that every pair of vertices belonging to different branches is connected with weight $\frac{\beta}{n_1n_2}$, where $\beta\ge0$ denotes the total cross-branch edge mass. The within-branch graph may be arbitrary, provided that each branch remains connected.

To isolate the deterministic distortion introduced by graph regularization from stochastic estimation noise, we study the mean regularized estimator $\overline V_\rho = \mathbb E[\widehat V_\rho].$ The following theorem gives an exact expression for the attenuation of the lineage contrast caused by graph regularization.
\begin{theorem}[Exact two-block branch-contrast attenuation]
\label{thm:two-block-attenuation}
Assume $\mathbb E[\Xi]=0$. Let
\[
\overline V_\rho = \mathbb E[\widehat V_\rho] = S_\rho V^\star.
\]
Then $\overline V_\rho$ remains constant within each branch, and its branch contrast satisfies
\begin{equation}
\Delta_{12}(\overline V_\rho) = a_\rho(\beta;n_1,n_2) \Delta^\star,
\label{eq:block-contrast-attenuation}
\end{equation}
where
\begin{equation}
a_\rho(\beta;n_1,n_2) = \frac{ 1 }{ 1+ \rho\beta \left( n_1^{-1} + n_2^{-1} \right) }.
\label{eq:block-attenuation-factor}
\end{equation}
Consequently,
\begin{equation}
\left\| \Delta_{12}(\overline V_\rho) - \Delta^\star \right\|_2^2 = \left[ \frac{ \rho\beta (n_1^{-1}+n_2^{-1}) }{ 1+ \rho\beta (n_1^{-1}+n_2^{-1}) } \right]^2 \|\Delta^\star\|_2^2.
\label{eq:block-contrast-bias}
\end{equation}
\end{theorem}

\begin{proof}
By the observation model $\widehat V^{(0)}=V^\star+\Xi$ and the assumption $\mathbb E[\Xi]=0$, we have
\[
\overline V_\rho = \mathbb E[\widehat V_\rho] = S_\rho\mathbb E[\widehat V^{(0)}] = S_\rho V^\star.
\]
Here the graph Laplacian $L$, and hence $S_\rho$, is treated as fixed.

Let $\mathcal U\subseteq\mathbb R^{N\times q}$ denote the subspace of matrices that are constant within each block. Since $V^\star$ is block-constant, $V^\star\in\mathcal U$.  For any $V\in\mathcal U$, let $u_1,u_2\in\mathbb R^q$ denote its two block values. If $i\in\mathcal I_1$, then
\begin{align*}
(LV)_i = \sum_{j\in\mathcal I_1}w_{ij}(u_1-u_1) + \sum_{j\in\mathcal I_2} \frac{\beta}{n_1n_2}(u_1-u_2) = \frac{\beta}{n_1}(u_1-u_2).
\end{align*}
Similarly, if $i\in\mathcal I_2$, then $(LV)_i = \frac{\beta}{n_2}(u_2-u_1)$. Thus $L\mathcal U\subseteq\mathcal U$. Consequently, $(I+\rho L)^{-1}$ also preserves $\mathcal U$, and hence $\overline V_\rho=S_\rho V^\star$ is constant within each block.

Let $u_1,u_2\in\mathbb R^q$ denote the two block values of $\overline V_\rho$. Since
\[
(I+\rho L)\overline V_\rho=V^\star,
\]
the equations on the two blocks reduce to
\begin{align}
u_1+\rho\frac{\beta}{n_1}(u_1-u_2) &= \mu_1, 
\label{eq:block-reduced-system-1}\\ 
u_2+\rho\frac{\beta}{n_2}(u_2-u_1) &= \mu_2.
\label{eq:block-reduced-system-2}
\end{align}
Subtracting Eq.~\eqref{eq:block-reduced-system-2} from Eq.~\eqref{eq:block-reduced-system-1} gives
\[
\left[ 1+\rho\beta \left( n_1^{-1}+n_2^{-1} \right) \right] (u_1-u_2) = \mu_1-\mu_2.
\]
Since
\[
\Delta_{12}(\overline V_\rho)=u_1-u_2 \qquad\text{and}\qquad \Delta^\star=\mu_1-\mu_2,
\]
Eq.~\eqref{eq:block-contrast-attenuation} follows.

Finally,
\[
\Delta_{12}(\overline V_\rho)-\Delta^\star = \left( a_\rho(\beta;n_1,n_2)-1 \right)\Delta^\star,
\]
and
\[
a_\rho(\beta;n_1,n_2)-1 = - \frac{ \rho\beta(n_1^{-1}+n_2^{-1}) }{ 1+\rho\beta(n_1^{-1}+n_2^{-1}) }.
\]
Taking squared Euclidean norms proves
Eq.~\eqref{eq:block-contrast-bias}.
\end{proof}

Theorem~\ref{thm:two-block-attenuation} shows that graph regularization does not merely reduce estimation variance; it also attenuates the true velocity contrast between daughter lineages through a simple multiplicative factor determined by the regularization strength and the effective cross-branch connectivity.

Next, we investigate the asymptotic consequence of persistent cross-branch connectivity.

\begin{theorem}[Persistent cross-connectivity implies persistent bias]
\label{thm:persistent-block-bias}
Consider a sequence of two-block problems indexed by $N$, with true branch contrast $\Delta_N^\star$. Suppose $\|\Delta_N^\star\|_2 \ge \delta_v > 0$, and define the effective cross-branch conductance $c_N = \beta_N ( n_{1,N}^{-1} + n_{2,N}^{-1} )$. If $\liminf_{N\rightarrow\infty} \rho_Nc_N \ge c_0 > 0$, then
\begin{equation}
\liminf_{N\rightarrow\infty} \left\| \Delta_{12} \left( \mathbb E [ \widehat V_{\rho_N} ] \right) - \Delta_N^\star \right\|_2^2 \ge \left( \frac{c_0}{1+c_0} \right)^2 \delta_v^2.
\label{eq:persistent-block-bias}
\end{equation}
\end{theorem}
\begin{proof}
For the \(N\)-th two-block problem, define
\[
x_N = \rho_N c_N = \rho_N\beta_N \left( n_{1,N}^{-1}+n_{2,N}^{-1} \right).
\]
By Eq.~\eqref{eq:block-contrast-bias},
\begin{align}
\left\| \Delta_{12} \left( \mathbb E[\widehat V_{\rho_N}] \right) - \Delta_N^\star \right\|_2^2 = \left( \frac{x_N}{1+x_N} \right)^2 \|\Delta_N^\star\|_2^2.
\label{eq:persistent-bias-proof-identity}
\end{align}
Since $\|\Delta_N^\star\|_2\geq\delta_v$, it follows that
\[
\left\| \Delta_{12} \left( \mathbb E[\widehat V_{\rho_N}] \right) - \Delta_N^\star \right\|_2^2 \geq \left( \frac{x_N}{1+x_N} \right)^2 \delta_v^2.
\]

The function
\[
f(x) = \left( \frac{x}{1+x} \right)^2, \qquad x\geq0,
\]
is continuous and nondecreasing. Therefore, the assumption
\[
\liminf_{N\to\infty}x_N = \liminf_{N\to\infty}\rho_Nc_N \geq c_0
\]
implies
\[
\liminf_{N\to\infty}f(x_N) \geq f(c_0) = \left( \frac{c_0}{1+c_0} \right)^2.
\]
Taking the lower limit in
Eq.~\eqref{eq:persistent-bias-proof-identity} yields
\[
\liminf_{N\to\infty} \left\| \Delta_{12} \left( \mathbb E[\widehat V_{\rho_N}] \right) - \Delta_N^\star \right\|_2^2 \geq \left( \frac{c_0}{1+c_0} \right)^2 \delta_v^2,
\]
which proves the claim.
\end{proof}

The theorem identifies the dimensionless quantity $\rho_Nc_N$ as the effective strength of cross-branch smoothing. If \(\rho_Nc_N\to0\), the branch-contrast bias vanishes in the two-block model. In contrast, if \(\rho_Nc_N\) remains bounded away from zero, increasing the number of cells does not eliminate the systematic attenuation of the daughter-branch contrast. Thus, asymptotic branch preservation requires either a decreasing regularization scale or a graph construction for which the effective cross-branch conductance tends to zero.

\subsubsection{Numerical validation of the two-block theory}
\label{subsec:two-block-numerical-validation}

We next use synthetic two-block data to illustrate the finite-sample and
asymptotic consequences of cross-branch graph coupling. The purpose of
these experiments is not to reproduce a specific biological dataset, but
to isolate the mechanism identified by
Theorem~\ref{thm:two-block-attenuation} and
Theorem~\ref{thm:persistent-block-bias} under a controlled graph model.

For the finite-dimensional experiments, we consider two branches
$\mathcal I_1$ and $\mathcal I_2$ of equal size with $n_1=n_2=15$. The true velocity field is block-constant, i.e.,
\[
V_i^\star = \mu_1, \qquad i\in\mathcal I_1,
\]
and
\[
V_i^\star = \mu_2, \qquad i\in\mathcal I_2,
\]
where $\mu_1=(0.90,0.40)^\top$ and $\mu_2=(-0.80,-0.50)^\top$.
Consequently, the true daughter-branch contrast is $\Delta^\star = \mu_1-\mu_2 \neq0$.

\begin{figure}[t]
\centering
\includegraphics[width=0.98\textwidth]{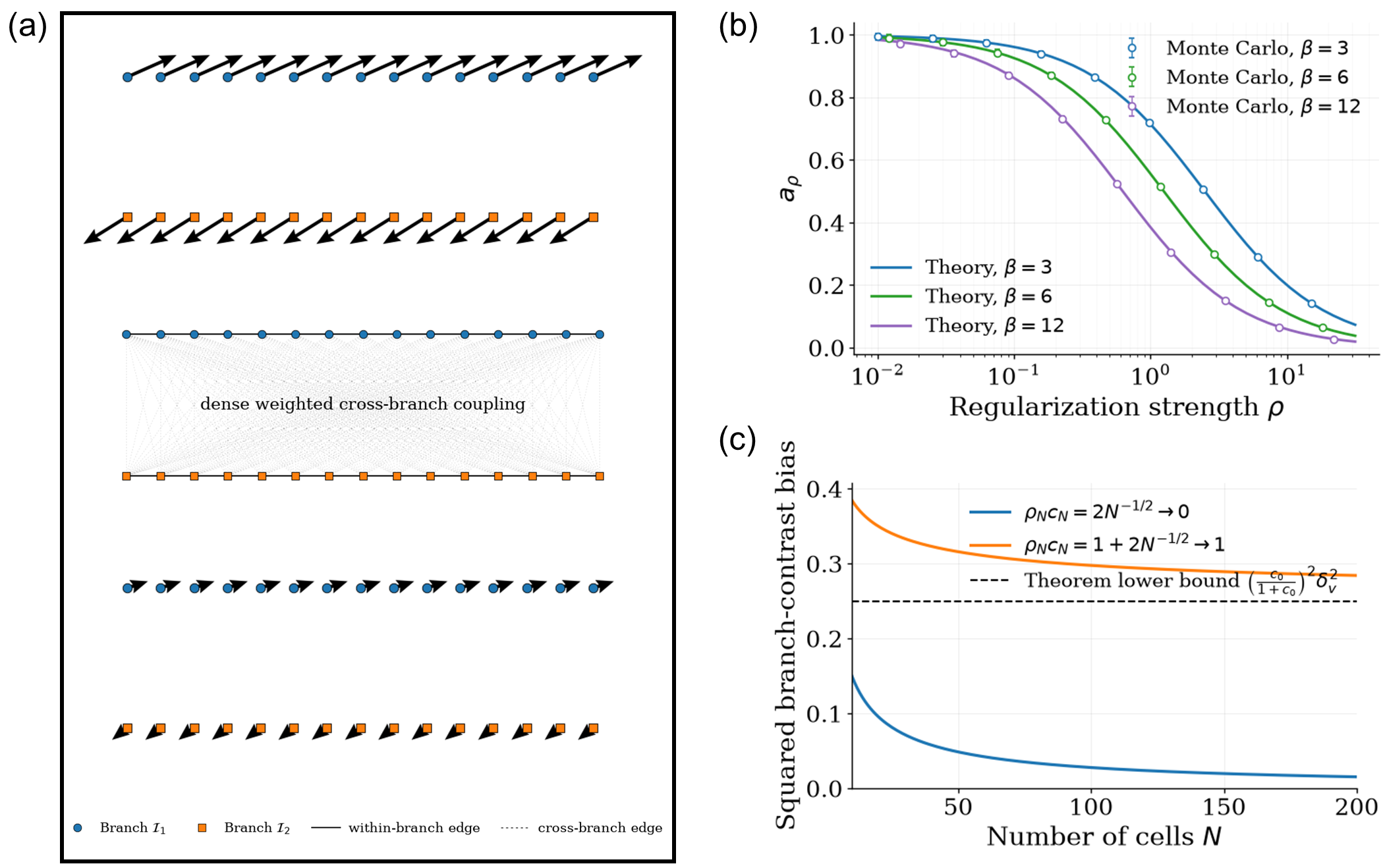}
\caption{
Numerical validation of the theoretical results for the weighted
two-block model.
(a) Illustration of the synthetic two-block model, including the true
block-constant velocity field (top), the weighted two-block graph
(middle), and the corresponding mean regularized velocity field
(bottom).
(b) Exact branch-contrast attenuation factor
$a_\rho(\beta;n_1,n_2)$
compared with Monte Carlo simulations for
$\beta\in\{3,6,12\}$.
(c) Squared branch-contrast bias under two asymptotic scaling regimes,
illustrating
Theorem~\ref{thm:persistent-block-bias}.
}
\label{fig:two-block-validation}
\end{figure}

Within each branch, consecutive cells are connected by unit-weight edges. Across branches, every cell in $\mathcal I_1$ is connected to every cell in $\mathcal I_2$ with common edge weight $\frac{\beta}{n_1n_2}$, so that the total cross-branch weight equals $\beta$. The resulting graph therefore consists of two connected within-branch subgraphs together with dense weighted cross-branch coupling. The effective cross-branch conductance is $c = \beta \left( \frac{1}{n_1} + \frac{1}{n_2} \right)$.

Figure~\ref{fig:two-block-validation}(a) summarizes the complete finite-dimensional construction. The top diagram shows the true block-constant velocity field before graph regularization. Cells in the two branches have distinct velocity vectors and hence a nonzero branch contrast. The middle diagram shows the weighted two-block graph. The solid edges represent within-branch connectivity, whereas the dense cross-branch edges couple the two daughter branches. The bottom diagram shows the mean regularized velocity field $\overline V_\rho = \mathbb E[\widehat V_\rho] = (I+\rho L)^{-1}V^\star$. The branch-specific velocity vectors are pulled toward one another after regularization, producing a visibly smaller daughter-branch contrast. For the parameters used in the illustration, we set $\beta=12$ and $\rho=1$. The numerical ratio $\frac{ \left\| \Delta_{12}(\overline V_\rho) \right\|_2 }{ \left\| \Delta^\star \right\|_2 }$ agrees, up to numerical precision, with the exact attenuation factor $a_\rho(\beta;n_1,n_2) = \frac{1}{ 1+\rho\beta \left( n_1^{-1}+n_2^{-1} \right) }.$ Thus, the visual reduction in branch separation in Figure~\ref{fig:two-block-validation}(a) is the deterministic cross-branch bias predicted by Theorem~\ref{thm:two-block-attenuation}.

To verify the attenuation formula in the presence of estimation noise, we next generate initial velocity estimates according to $\widehat V^{(0)} = V^\star+\Xi$, where the entries of $\Xi$ are independent centered Gaussian random variables with standard deviation $\sigma=0.60$.

For each value of $\rho$, we compute $\widehat V_\rho = (I+\rho L)^{-1}\widehat V^{(0)}$ over $500$ independent Monte Carlo replications. For each replication, the regularized branch contrast is projected onto the direction of the true contrast, and the resulting scalar attenuation factors are averaged.

Figure~\ref{fig:two-block-validation}(b) compares these Monte Carlo estimates with the exact expression $a_\rho(\beta;n_1,n_2)$ for $\beta\in\{3,6,12\}$. The empirical estimates closely follow the analytical curves throughout the full range of regularization strengths. The agreement confirms that the theoretical attenuation factor describes the mean effect of graph regularization even when the initial velocity field is contaminated by substantial noise. Moreover, increasing $\beta$ shifts the attenuation curve toward smaller values of $\rho$. Thus, stronger cross-branch connectivity causes biologically meaningful daughter-branch contrasts to be lost under weaker regularization.

Finally, we examine the asymptotic statement in Theorem~\ref{thm:persistent-block-bias}. For clarity, we take $\|\Delta_N^\star\|_2=\delta_v=1$ and compare two sequences of effective cross-branch smoothing strengths. In the vanishing regime, we set $\rho_Nc_N = 2N^{-1/2} \to 0.$ The corresponding squared branch-contrast bias is $\left( \frac{2N^{-1/2}}{1+2N^{-1/2}} \right)^2$, which converges to zero. In the persistent regime, we set $\rho_Nc_N = 1+2N^{-1/2} \to 1$. The corresponding bias is $\left( \frac{1+2N^{-1/2}}{2+2N^{-1/2}} \right)^2$, which converges to $\left( \frac{1}{1+1} \right)^2 = \frac14.$

These two behaviors are shown in Figure~\ref{fig:two-block-validation}(c). When $\rho_Nc_N\to0$, increasing the number of cells eliminates the branch-contrast bias. In contrast, when $\rho_Nc_N\to c_0>0$, the bias approaches the positive lower bound $\left( \frac{c_0}{1+c_0} \right)^2 \delta_v^2.$ Hence, increasing sample size alone does not guarantee branch preservation. Consistency requires the combined strength of graph regularization and effective cross-branch conductance to vanish. These experiments therefore support both the finite-dimensional attenuation formula and the persistent-bias mechanism established above.

\subsection{Bifurcation geometry and asymptotic branch error}
\label{sec:bifurcation-results}

We now formulate an explicit geometric model showing how a proximity-based cell-state graph can retain harmful cross-branch connections near a developmental bifurcation. We then relate this geometric phenomenon to a deterministic quotient model in which persistent cross-branch conductance produces a nonvanishing branch-contrast error.

\subsubsection{Persistent cross-branch connectivity and asymptotic branch error}
\label{subsec:bifurcation-connectivity-error}

Let the latent lineage consist of a parent branch
\(\mathcal M_0\) and two daughter branches
\(\mathcal M_+\) and \(\mathcal M_-\), embedded in
\(\mathbb R^p\). We parameterize the daughter branches by
\begin{equation}
\gamma_{\pm,N}(s) =
\begin{pmatrix}
s\\
\pm\varepsilon_N+\theta_\pm s\\
0_{p-2}
\end{pmatrix},
\qquad
s\in[0,\ell],
\label{eq:y-branch-parameterization}
\end{equation}
where $\varepsilon_N\geq0, \theta_+,\theta_-\in\mathbb R$, and $\theta_+\neq\theta_-$. The quantity \(2\varepsilon_N\) is the ambient separation between the two daughter branches at \(s=0\), whereas \(\theta_+\) and \(\theta_-\) determine their local tangent directions.

For each cell \(i\), let $B_i\in\{+,-\}$ denote its daughter-branch label, with $\mathbb P(B_i=+)=\pi_+$ and $\mathbb P(B_i=-)=\pi_-$, where $\pi_+>0, \pi_->0$, and $\pi_++\pi_-=1$. Conditional on \(B_i=b\), let \(S_i\) have density \(f_b\) on \([0,\ell]\), and define $X_i=\gamma_{b,N}(S_i)$. Assume that there exist constants $0<f_{\min}\leq f_{\max}<\infty$ and \(s_0>0\) such that
\begin{equation}
f_{\min} \leq f_b(s) \leq f_{\max}, \qquad s\in[0,s_0], \quad b\in\{+,-\}.
\label{eq:density-bounds-near-branch-point}
\end{equation}

Let \(h_N>0\) be the graph bandwidth. We construct the standard
bandwidth kernel graph using $w_{ij}^{\mathrm{std}} = K\left( \frac{\|X_i-X_j\|_2}{h_N} \right)$, where \(K:[0,\infty)\to[0,\infty)\) satisfies
\begin{equation}
0\leq K(r)\leq K_{\max}<\infty,
\qquad
r\geq0,
\label{eq:kernel-upper-bound}
\end{equation}
and
\begin{equation}
K(r)\geq k_0>0,
\qquad
0\leq r\leq1.
\label{eq:kernel-lower-bound}
\end{equation}
Thus, any pair of cells separated by at most \(h_N\) receives an edge weight bounded below by \(k_0\).

For a constant \(a>0\), define the local bifurcation index set $\mathcal I_N^{\mathrm{loc}} = \left\{ i: S_i\in[0,ah_N] \right\}$. This parameter-based definition avoids ambiguity caused by possible overlap of the two embedded daughter branches. Define the local cross-branch edge mass by
\begin{equation}
C_N^{\mathrm{cross}} =
\sum_{\substack{
1\leq i<j\leq N\\
i,j\in\mathcal I_N^{\mathrm{loc}}\\
B_i\neq B_j
}}
w_{ij}^{\mathrm{std}},
\label{eq:cross-edge-mass}
\end{equation}
and the corresponding within-daughter edge mass by
\begin{equation}
C_N^{\mathrm{within}} =
\sum_{\substack{
1\leq i<j\leq N\\
i,j\in\mathcal I_N^{\mathrm{loc}}\\
B_i=B_j
}}
w_{ij}^{\mathrm{std}}.
\label{eq:within-edge-mass}
\end{equation}

The following result shows that increasing the number of sampled cells does not necessarily make the local proximity graph branch preserving.

\begin{theorem}[Persistent cross-branch connectivity]
\label{thm:persistent-cross-edges-complete}
Assume that $h_N\to 0, Nh_N\to \infty$, and $\frac{\varepsilon_N}{h_N} \to c_\varepsilon \in \left[0,\frac12\right)$. Then there exist constants \(a>0\) and \(c_{\mathrm{cross}}>0\), independent of \(N\), such that
\begin{equation}
\mathbb P \left[ \frac{ C_N^{\mathrm{cross}} }{ C_N^{\mathrm{within}} + C_N^{\mathrm{cross}} } \geq c_{\mathrm{cross}} \right] \to 1.
\label{eq:persistent-cross-edge-ratio-complete}
\end{equation}
In particular, the fraction of local graph mass assigned to cross-branch edges does not converge to zero.
\end{theorem}

\begin{proof}
See Appendix~\ref{secA2}.
\end{proof}

Theorem~\ref{thm:persistent-cross-edges-complete} is a geometric result -- it shows that the observed graph can retain a positive proportion of cross-branch edge mass inside a shrinking neighborhood of the bifurcation. To describe the statistical consequence of such connectivity, we next introduce an exact local quotient model.

Let the daughter-branch cells be partitioned into local bins 
\[
\mathcal I_{r,+,N}, \qquad \mathcal I_{r,-,N}, \qquad r=1,\ldots,m_N,
\]
with $\mathcal I_{r,+,N} \cap \mathcal I_{r,-,N} = \varnothing$. Define
\[
n_{r,+,N} = |\mathcal I_{r,+,N}|, \qquad n_{r,-,N} = |\mathcal I_{r,-,N}|, \qquad n_{r,N} = n_{r,+,N}+n_{r,-,N},
\]
and assume $n_{r,+,N}\geq1$ and $n_{r,-,N}\geq1$. After ordering the vertices in bin \(r\) so that the \(+\)-branch vertices precede the \(-\)-branch vertices, define the local block-mean operator
\begin{equation}
M_{r,N} =
\begin{bmatrix}
n_{r,+,N}^{-1}\mathbf 1_{n_{r,+,N}}^\top & 0
\\[1mm]
0 & n_{r,-,N}^{-1}\mathbf 1_{n_{r,-,N}}^\top
\end{bmatrix}
\in \mathbb R^{2\times n_{r,N}},
\label{eq:local-block-mean-operator}
\end{equation}
and the local lifting operator
\begin{equation}
J_{r,N} =
\begin{bmatrix}
\mathbf 1_{n_{r,+,N}} & 0
\\[1mm]
0 & \mathbf 1_{n_{r,-,N}}
\end{bmatrix}
\in
\mathbb R^{n_{r,N}\times2}.
\label{eq:local-block-lifting-operator}
\end{equation}
Then, we have $M_{r,N}J_{r,N}=I_2$.

Let $d=(1, -1)^\top$. For any $V_r\in\mathbb R^{n_{r,N}\times q}$, define the local daughter-branch contrast by $\Delta_{r,N}(V_r) = d^\top M_{r,N}V_r \in \mathbb R^{1\times q}$. We identify this row vector with the corresponding element of \(\mathbb R^q\) when taking Euclidean norms.

Assume that the true velocity field is block-constant within each local
bin, i.e.,
\begin{equation}
V_{r,N}^\star = J_{r,N}U_{r,N}^\star,
\qquad
U_{r,N}^\star =
\begin{bmatrix}
\mu_{r,+,N}^\top\\[1mm]
\mu_{r,-,N}^\top
\end{bmatrix}
\in
\mathbb R^{2\times q},
\label{eq:local-block-constant-signal}
\end{equation}
where $\mu_{r,+,N}, \mu_{r,-,N} \in \mathbb R^q$. The true local contrast is
\begin{equation}
\Delta_{r,N}^\star = \Delta_{r,N}(V_{r,N}^\star) = \mu_{r,+,N} - \mu_{r,-,N}.
\label{eq:true-local-contrast}
\end{equation}

Let \(L_{r,N}\in\mathbb R^{n_{r,N}\times n_{r,N}}\) denote the
Laplacian of the weighted subgraph induced by $\mathcal I_{r,+,N} \cup \mathcal I_{r,-,N}$. Edges connecting bin \(r\) to cells outside that bin are not included
in \(L_{r,N}\). Thus, we have $L_{r,N}=L_{r,N}^\top$ and $L_{r,N}\succeq0$.

Assume that the block-constant subspace $\operatorname{range}(J_{r,N})$ is invariant under \(L_{r,N}\), with
\begin{equation}
L_{r,N}J_{r,N} = J_{r,N}\overline L_{r,N},
\label{eq:exact-quotient-intertwining}
\end{equation}
where
\begin{equation}
\overline L_{r,N} =
\begin{bmatrix}
\beta_{r,N}/n_{r,+,N} & -\beta_{r,N}/n_{r,+,N}
\\[1mm]
-\beta_{r,N}/n_{r,-,N} & \beta_{r,N}/n_{r,-,N}
\end{bmatrix}
\in \mathbb R^{2\times2}
\label{eq:local-quotient-laplacian}
\end{equation}
for some $\beta_{r,N}\geq0$. Here \(\beta_{r,N}\) is the total cross-branch edge mass in bin \(r\). The exact intertwining condition requires uniform cross-branch weighted degree within each of the two local blocks. Within-branch edges may otherwise be arbitrary because they vanish on block-constant signals.

Define the effective local quotient conductance by $c_{r,N} = \beta_{r,N} \left( n_{r,+,N}^{-1} + n_{r,-,N}^{-1} \right)$. Throughout the remainder of this result, the sampled cell states, the bin partition, the graph, and the loss weights are treated as fixed, and all expectations are conditional on these quantities. Assume the conditionally unbiased observation model
\begin{equation}
\widehat V_{r,N}^{(0)} = V_{r,N}^\star+\Xi_{r,N},
\qquad
\mathbb E[\Xi_{r,N}]=0.
\label{eq:local-observation-model}
\end{equation}
For a regularization parameter $\rho_N\geq0$, define
\begin{equation}
\widehat V_{r,N,\rho_N} = (I_{n_{r,N}}+\rho_NL_{r,N})^{-1} \widehat V_{r,N}^{(0)}.
\label{eq:local-regularized-estimator}
\end{equation}

Let $\omega_{r,N}\geq0$, where $r=1,\ldots,m_N$, be deterministic weights satisfying $\sum_{r=1}^{m_N}\omega_{r,N}=1$. Define the weighted branch-contrast loss by
\begin{equation}
\mathcal L_{\mathrm{contrast},N} \left( \widehat V_{\rho_N}, V_N^\star \right) = \sum_{r=1}^{m_N} \omega_{r,N} \left\| \Delta_{r,N} \left( \widehat V_{r,N,\rho_N} \right) - \Delta_{r,N}^\star \right\|_2^2.
\label{eq:weighted-contrast-loss}
\end{equation}

\begin{theorem}[Nonvanishing branch-contrast error under persistent quotient conductance]
\label{thm:rigorous-quotient-inconsistency}
Suppose that there exist constants $\omega_0>0, \delta_v>0$, and $c_0>0$, such that for all sufficiently large \(N\), there exists an index set $\mathcal R_N \subseteq \{1,\ldots,m_N\}$ satisfying
\begin{align}
\sum_{r\in\mathcal R_N}\omega_{r,N}
&\geq
\omega_0,
\label{eq:positive-affected-mass}
\\
\|\Delta_{r,N}^\star\|_2
&\geq
\delta_v,
\qquad
r\in\mathcal R_N,
\label{eq:uniform-local-contrast}
\\
\rho_Nc_{r,N}
&\geq
c_0,
\qquad
r\in\mathcal R_N.
\label{eq:uniform-persistent-conductance}
\end{align}
Then, for all sufficiently large \(N\),
\begin{equation}
\mathbb E \left[ \mathcal L_{\mathrm{contrast},N} \left( \widehat V_{\rho_N}, V_N^\star \right) \right] \geq \omega_0 \left( \frac{c_0}{1+c_0} \right)^2 \delta_v^2.
\label{eq:finite-n-contrast-lower-bound}
\end{equation}
Consequently,
\begin{equation}
\liminf_{N\to\infty} \mathbb E \left[ \mathcal L_{\mathrm{contrast},N} \left( \widehat V_{\rho_N}, V_N^\star \right) \right] \geq \omega_0 \left( \frac{c_0}{1+c_0} \right)^2 \delta_v^2 > 0.
\label{eq:asymptotic-contrast-lower-bound}
\end{equation}
Hence the estimator is not mean-square consistent with respect to the weighted branch-contrast loss.
\end{theorem}

\begin{proof}
See Appendix~\ref{secA3}.
\end{proof}

The two preceding theorems separate the geometric and statistical components of the inconsistency mechanism. Theorem~\ref{thm:persistent-cross-edges-complete} shows that a bandwidth proximity graph may retain a nonvanishing fraction of cross-branch edge mass near the bifurcation. Theorem~\ref{thm:rigorous-quotient-inconsistency} shows that whenever this connectivity induces a persistent effective quotient conductance on a nonnegligible collection of local bins, the expected branch-contrast loss remains bounded away from zero.

The quotient result is exact and deterministic. A random proximity graph will generally satisfy the intertwining relation Eq.~\eqref{eq:exact-quotient-intertwining} only approximately. Transferring the conclusion to the full random graph therefore requires an additional approximation or resolvent-perturbation argument controlling the discrepancy between the empirical local graph and its quotient representation.

\begin{appendices}
\section{Proof of Theorem~\ref{thm:exact-spectral-risk}}
\label{secA1}

\begin{proof}
Define $A_\rho = (S_\rho-I)V^\star$ and $Z_\rho = S_\rho\Xi$. Then, we have
\[
\widehat V_\rho-V^\star = A_\rho+Z_\rho.
\]
Using the Frobenius inner product
\[
\langle A,B\rangle_F = \operatorname{tr}(A^\top B),
\]
we obtain
\begin{align}
\left\| \widehat V_\rho-V^\star \right\|_F^2 & =
\|A_\rho+Z_\rho\|_F^2 \nonumber\\
& = \|A_\rho\|_F^2 + \|Z_\rho\|_F^2 + 2\langle A_\rho,Z_\rho\rangle_F
\nonumber\\
& = \left\| (S_\rho-I)V^\star \right\|_F^2 + \left\| S_\rho\Xi \right\|_F^2 +
2\operatorname{tr} \left( \bigl[(S_\rho-I)V^\star\bigr]^\top S_\rho\Xi \right).
\label{eq:risk-expansion-proof}
\end{align}

Because $L$ is symmetric, $S_\rho=(I+\rho L)^{-1}$ is also symmetric. Consequently, the cross term can be written as
\begin{align}
\operatorname{tr} \left( \bigl[(S_\rho-I)V^\star\bigr]^\top S_\rho\Xi \right) 
& = \operatorname{tr} \left( (V^\star)^\top (S_\rho-I)^\top S_\rho\Xi \right) \nonumber\\
& = \operatorname{tr}\left( (V^\star)^\top (S_\rho-I) S_\rho\Xi \right).
\end{align}
The matrices \(V^\star\), \(S_\rho\), and \(L\) are fixed under the conditional framework. Therefore, by linearity of expectation,
\begin{align}
\mathbb{E} \left[ \operatorname{tr} \left( (V^\star)^\top (S_\rho-I) S_\rho\Xi \right) \right]
= \operatorname{tr} \left( (V^\star)^\top (S_\rho-I) S_\rho \mathbb E[\Xi] \right)= 0.
\label{eq:cross-term-zero-proof}
\end{align}
Taking expectations in \eqref{eq:risk-expansion-proof} therefore gives
\begin{equation}
\mathcal R_{\mathrm{global}}(\rho;L,V^\star) 
= \left\| (S_\rho-I)V^\star \right\|_F^2 
+ \mathbb E \left[ \left\| S_\rho\Xi \right\|_F^2 \right].
\label{eq:bias-variance-basic-proof}
\end{equation}

We first calculate the squared bias term. Since
\[
S_\rho = \Phi(I+\rho\Lambda)^{-1}\Phi^\top,
\]
we have
\begin{align}
S_\rho-I 
&= \Phi(I+\rho\Lambda)^{-1}\Phi^\top - \Phi I\Phi^\top \nonumber\\ 
&= \Phi \left[ (I+\rho\Lambda)^{-1}-I \right] \Phi^\top.
\label{eq:smoother-minus-identity-proof}
\end{align}
It follows that
\begin{align}
(S_\rho-I)V^\star 
&= \Phi \left[ (I+\rho\Lambda)^{-1}-I \right] \Phi^\top V^\star \nonumber\\ 
&= \Phi \left[ (I+\rho\Lambda)^{-1}-I \right] \widetilde V^\star.
\end{align}
Since multiplication by the orthogonal matrix \(\Phi\) preserves the Frobenius norm,
\begin{align}
\left\| (S_\rho-I)V^\star \right\|_F^2 = \left\| \left[ (I+\rho\Lambda)^{-1}-I \right] \widetilde V^\star \right\|_F^2.
\label{eq:bias-orthogonal-invariance-proof}
\end{align}
The diagonal entries of $(I+\rho\Lambda)^{-1}-I$ are $\frac{1}{1+\rho\lambda_k}-1 = -\frac{\rho\lambda_k}{1+\rho\lambda_k}$. Therefore, the \(k\)-th row of $\left[ (I+\rho\Lambda)^{-1}-I \right] \widetilde V^\star$ is $-\frac{\rho\lambda_k} {1+\rho\lambda_k} \widetilde V^\star_k$. Summing the squared Euclidean norms of the rows gives
\begin{align}
\left\| (S_\rho-I)V^\star \right\|_F^2 
&= \sum_{k=1}^N \left\| -\frac{\rho\lambda_k} {1+\rho\lambda_k} \widetilde V^\star_k \right\|_2^2 \nonumber\\ 
&= \sum_{k=1}^N \left( \frac{\rho\lambda_k} {1+\rho\lambda_k} \right)^2 \left\| \widetilde V^\star_k \right\|_2^2.
\label{eq:exact-bias-term-proof}
\end{align}

We next calculate the variance term. Using $\|A\|_F^2 = \operatorname{tr}(A^\top A)$, and the symmetry of \(S_\rho\), we obtain
\begin{align}
\left\| S_\rho\Xi \right\|_F^2 
&= \operatorname{tr} \left[ (S_\rho\Xi)^\top (S_\rho\Xi) \right] \nonumber\\ 
&= \operatorname{tr} \left( \Xi^\top S_\rho^\top S_\rho\Xi \right) \nonumber\\ 
&= \operatorname{tr} \left( \Xi^\top S_\rho^2\Xi \right).
\label{eq:variance-trace-form-proof}
\end{align}
For any \(N\times q\) matrix \(A\) and any \(N\times N\) matrix \(M\), the vectorization identity
\[
\operatorname{tr}(A^\top M A) = \operatorname{vec}(A)^\top (I_q\otimes M) \operatorname{vec}(A)
\]
holds. Applying this identity with \(A=\Xi\) and \(M=S_\rho^2\) yields
\begin{equation}
\left\| S_\rho\Xi \right\|_F^2 
= \operatorname{vec}(\Xi)^\top \left( I_q\otimes S_\rho^2 \right) \operatorname{vec}(\Xi).
\label{eq:variance-vectorization-proof}
\end{equation}

Let $\xi = \operatorname{vec}(\Xi)$. Because \(\mathbb E[\xi]=0\) and $\operatorname{Cov}(\xi) = \sigma^2I_{Nq}$, the standard quadratic-form identity 
\[ \mathbb E[\xi^\top A\xi] = \operatorname{tr} \left( A\operatorname{Cov}(\xi) \right) + \mathbb E[\xi]^\top A\mathbb E[\xi]
\]
gives
\begin{align}
\mathbb E \left[ \left\| S_\rho\Xi \right\|_F^2 \right] 
&= \operatorname{tr} \left[ \left( I_q\otimes S_\rho^2 \right) \operatorname{Cov}(\xi) \right] \nonumber\\ 
&= \sigma^2 \operatorname{tr} \left( I_q\otimes S_\rho^2 \right).
\label{eq:variance-covariance-proof}
\end{align}
Using the trace identity
\[
\operatorname{tr}(A\otimes B) = \operatorname{tr}(A)\operatorname{tr}(B),
\]
we obtain
\begin{align}
\operatorname{tr} \left( I_q\otimes S_\rho^2 \right) 
&= \operatorname{tr}(I_q) \operatorname{tr}(S_\rho^2) \nonumber\\ 
&= q\operatorname{tr}(S_\rho^2).
\end{align}
Hence
\begin{equation}
\mathbb E \left[ \left\| S_\rho\Xi \right\|_F^2 \right] = q\sigma^2\operatorname{tr}(S_\rho^2).
\label{eq:variance-trace-result-proof}
\end{equation}

Finally, from the spectral representation of \(S_\rho\),
\begin{align}
S_\rho^2 = \Phi(I+\rho\Lambda)^{-2}\Phi^\top.
\end{align}
By invariance of the trace under orthogonal similarity transformations,
\begin{align}
\operatorname{tr}(S_\rho^2) 
&= \operatorname{tr} \left[ \Phi(I+\rho\Lambda)^{-2}\Phi^\top \right] \nonumber\\ 
&= \operatorname{tr} \left[ (I+\rho\Lambda)^{-2} \right] \nonumber\\ 
&= \sum_{k=1}^N \frac{1} {(1+\rho\lambda_k)^2}.
\label{eq:smoother-trace-spectral-proof}
\end{align}
Therefore, we have
\begin{equation}
\mathbb E \left[ \left\| S_\rho\Xi \right\|_F^2 \right] = q\sigma^2 \sum_{k=1}^N \frac{1} {(1+\rho\lambda_k)^2}.
\label{eq:exact-variance-term-proof}
\end{equation}

Substituting Eqs.~\eqref{eq:exact-bias-term-proof} and \eqref{eq:exact-variance-term-proof} into Eq.~\eqref{eq:bias-variance-basic-proof} proves Eq.~\eqref{eq:exact-spectral-risk}.
\end{proof}

\section{Proof of Theorem~\ref{thm:persistent-cross-edges-complete}}
\label{secA2}
\begin{proof}
Let $T = |\theta_+|+|\theta_-|$. Since $c_\varepsilon<\frac12$, we may choose a constant \(\eta\) satisfying
\begin{equation}
2c_\varepsilon<\eta<1.
\label{eq:eta-choice}
\end{equation}
Because $\frac{2\varepsilon_N}{h_N} \longrightarrow 2c_\varepsilon$, there exists \(N_0\) such that
\begin{equation}
\frac{2\varepsilon_N}{h_N} \leq\eta, \qquad N\geq N_0.
\label{eq:epsilon-upper-bound}
\end{equation}

Choose \(a>0\) sufficiently small that
\begin{equation}
a^2+(\eta+aT)^2<1.
\label{eq:a-geometric-condition}
\end{equation}
Such a choice is possible because \(\eta<1\). Since \(h_N\to0\), we also have $ah_N\leq s_0$ for all sufficiently large \(N\).

We first show that every pair consisting of one cell from each daughter branch inside \(\mathcal U_N\) lies within graph distance \(h_N\). Let $s,t\in[0,ah_N]$. Using the branch parameterization,
\begin{align}
\gamma_{+,N}(s)-\gamma_{-,N}(t) =
\begin{pmatrix}
s-t\\
2\varepsilon_N+\theta_+s-\theta_-t\\
0_{p-2}
\end{pmatrix}.
\end{align}
Consequently,
\begin{align}
\left\| \gamma_{+,N}(s)-\gamma_{-,N}(t) \right\|_2^2 = (s-t)^2 + \left( 2\varepsilon_N+\theta_+s-\theta_-t \right)^2.
\label{eq:cross-distance-exact}
\end{align}
Since \(s,t\in[0,ah_N]\), $|s-t|\leq ah_N$, and
\[
\left| 2\varepsilon_N+\theta_+s-\theta_-t \right| \leq 2\varepsilon_N + a h_N \left( |\theta_+|+|\theta_-| \right).
\]
For sufficiently large \(N\), Eq.~\eqref{eq:epsilon-upper-bound} therefore gives
\begin{align}
\left\| \gamma_{+,N}(s)-\gamma_{-,N}(t) \right\|_2^2 \leq h_N^2 \left[ a^2+(\eta+aT)^2 \right] < h_N^2,
\label{eq:all-local-cross-pairs-close}
\end{align}
where the final inequality follows from Eq.~\eqref{eq:a-geometric-condition}. Hence
\begin{equation}
\left\| \gamma_{+,N}(s)-\gamma_{-,N}(t) \right\|_2 < h_N.
\label{eq:local-cross-pair-bandwidth}
\end{equation}
By the kernel lower bound Eq.~\eqref{eq:kernel-lower-bound}, every such cross-branch pair has weight at least \(k_0\).

Define the local daughter-branch cell counts
\begin{equation}
Z_{b,N} = \sum_{i=1}^N \mathbf 1 \left\{ B_i=b,\, S_i\in[0,ah_N] \right\}, \qquad b\in\{+,-\}.
\label{eq:local-branch-count}
\end{equation}
Then $Z_{b,N} \sim \operatorname{Binomial}(N,p_{b,N})$, where
\begin{equation}
p_{b,N} = \pi_b \int_0^{ah_N}f_b(s)\,ds.
\label{eq:local-branch-probability}
\end{equation}
The density bounds imply
\begin{equation}
\pi_b f_{\min}ah_N \leq p_{b,N} \leq \pi_b f_{\max}ah_N.
\label{eq:local-probability-bounds}
\end{equation}
In particular,
\[
Np_{b,N} \geq \pi_b f_{\min}aNh_N \longrightarrow\infty.
\]

A standard multiplicative Chernoff bound gives
\begin{equation}
\mathbb P \left[ Z_{b,N} < \frac12Np_{b,N} \right] \leq \exp \left( -\frac18Np_{b,N} \right) \longrightarrow0.
\label{eq:chernoff-lower}
\end{equation}
Therefore, with probability tending to one,
\begin{equation}
Z_{b,N} \geq \frac12 \pi_b f_{\min}aNh_N, \qquad b\in\{+,-\}.
\label{eq:local-count-lower}
\end{equation}

Let $Z_N=Z_{+,N}+Z_{-,N}$ be the total number of cells in \(\mathcal U_N\). Since
\begin{align}
\mathbb E[Z_N] 
&= N \sum_{b\in\{+,-\}} \pi_b \int_0^{ah_N}f_b(s)\,ds \nonumber\\ 
&\leq Naf_{\max}h_N,
\end{align}
another Chernoff bound implies
\begin{equation}
\mathbb P \left[ Z_N>2af_{\max}Nh_N \right] \longrightarrow0.
\label{eq:local-total-count-upper}
\end{equation}
Thus, with probability tending to one,
\begin{equation}
Z_N \leq 2af_{\max}Nh_N.
\label{eq:local-count-upper-event}
\end{equation}

On the intersection of the high-probability events Eq.~\eqref{eq:local-count-lower} and Eq.~\eqref{eq:local-count-upper-event}, every pair formed by one local \(+\)-branch cell and one local \(-\)-branch cell has weight at least \(k_0\). Consequently,
\begin{align}
C_N^{\mathrm{cross}} \geq k_0Z_{+,N}Z_{-,N} \geq \frac{ k_0\pi_+\pi_-f_{\min}^2a^2 }{4} N^2h_N^2.
\label{eq:cross-mass-lower-bound}
\end{align}

On the other hand, all edge weights are bounded above by \(K_{\max}\). The total number of unordered pairs of cells in \(\mathcal U_N\) is at most
\[
\binom{Z_N}{2}
\leq
\frac12Z_N^2.
\]
Therefore,
\begin{align}
C_N^{\mathrm{within}} + C_N^{\mathrm{cross}} 
&\leq K_{\max}\binom{Z_N}{2} \nonumber\\ 
&\leq \frac{K_{\max}}{2}Z_N^2 \nonumber\\ 
&\leq 2K_{\max}a^2f_{\max}^2 N^2h_N^2.
\label{eq:total-mass-upper-bound}
\end{align}

Combining Eq.~\eqref{eq:cross-mass-lower-bound} and Eq.~\eqref{eq:total-mass-upper-bound} yields
\begin{align}
\frac{ C_N^{\mathrm{cross}} }{ C_N^{\mathrm{within}} + C_N^{\mathrm{cross}} } &\geq \frac{ k_0\pi_+\pi_-f_{\min}^2 }{ 8K_{\max}f_{\max}^2 }
\label{eq:cross-ratio-explicit-lower-bound}
\end{align}
with probability tending to one.

Hence the conclusion holds with, for example,
\begin{equation}
c_{\mathrm{cross}} = \min \left\{ \frac12,\, \frac{ k_0\pi_+\pi_-f_{\min}^2 }{ 8K_{\max}f_{\max}^2 } \right\} >0.
\label{eq:cross-constant-definition}
\end{equation}
This proves Eq.~\eqref{eq:persistent-cross-edge-ratio-complete}.
\end{proof}

\section{Proof of Theorem~\ref{thm:rigorous-quotient-inconsistency}}
\label{secA3}

\begin{proof}
Fix \(N\) and a bin \(r\). By the observation model Eq.~\eqref{eq:local-observation-model}, we have
\[
\widehat V_{r,N}^{(0)} = V_{r,N}^\star+\Xi_{r,N}, \qquad \mathbb E[\Xi_{r,N}]=0.
\]
Since \(L_{r,N}\) is treated as fixed under the conditional expectation,
\begin{equation}
\mathbb E \left[ \widehat V_{r,N,\rho_N} \right] = (I_{n_{r,N}}+\rho_NL_{r,N})^{-1} V_{r,N}^\star.
\label{eq:expected-local-estimator}
\end{equation}

By Eq.~\eqref{eq:local-block-constant-signal}, we have $V_{r,N}^\star = J_{r,N}U_{r,N}^\star$. The intertwining identity Eq.~\eqref{eq:exact-quotient-intertwining} gives
\begin{align}
(I_{n_{r,N}}+\rho_NL_{r,N})J_{r,N} 
&= J_{r,N} + \rho_NL_{r,N}J_{r,N} \nonumber\\ 
&= J_{r,N} + \rho_NJ_{r,N}\overline L_{r,N} \nonumber\\ 
&= J_{r,N} (I_2+\rho_N\overline L_{r,N}).
\label{eq:resolvent-intertwining-step}
\end{align}

Because $L_{r,N}\succeq0$ and $\rho_N\geq0$, the matrix $I_{n_{r,N}}+\rho_NL_{r,N}$ is invertible. Moreover, the quotient matrix \(\overline L_{r,N}\) has eigenvalues $0$ and $c_{r,N} = \beta_{r,N} \left( n_{r,+,N}^{-1} + n_{r,-,N}^{-1} \right) \geq0$. Therefore, matrix $I_2+\rho_N\overline L_{r,N}$ has eigenvalues $1$ and $1+\rho_Nc_{r,N}>0$, and is also invertible.

Multiplying Eq.~\eqref{eq:resolvent-intertwining-step} from the left by \((I_{n_{r,N}}+\rho_NL_{r,N})^{-1}\) and from the right by \((I_2+\rho_N\overline L_{r,N})^{-1}\) yields
\begin{equation}
(I_{n_{r,N}}+\rho_NL_{r,N})^{-1}J_{r,N} = J_{r,N} (I_2+\rho_N\overline L_{r,N})^{-1}.
\label{eq:resolvent-intertwining}
\end{equation}
Combining Eq.~\eqref{eq:expected-local-estimator}, Eq.~\eqref{eq:local-block-constant-signal}, and Eq.~\eqref{eq:resolvent-intertwining}, we obtain
\begin{align}
\mathbb E \left[ \widehat V_{r,N,\rho_N} \right] 
&= (I_{n_{r,N}}+\rho_NL_{r,N})^{-1} J_{r,N}U_{r,N}^\star \nonumber\\ 
&= J_{r,N} (I_2+\rho_N\overline L_{r,N})^{-1} U_{r,N}^\star.
\label{eq:expected-estimator-quotient-form}
\end{align}

Let
\[
U_{r,N,\rho_N} = (I_2+\rho_N\overline L_{r,N})^{-1} U_{r,N}^\star =
\begin{bmatrix}
u_{r,+,N}^\top
\\[1mm]
u_{r,-,N}^\top
\end{bmatrix},
\]
where
\[
u_{r,+,N}, u_{r,-,N} \in \mathbb R^q.
\]
Then
\[
(I_2+\rho_N\overline L_{r,N}) U_{r,N,\rho_N} = U_{r,N}^\star,
\]
which is equivalent to
\begin{align}
u_{r,+,N} + \rho_N \frac{\beta_{r,N}}{n_{r,+,N}} \left( u_{r,+,N}-u_{r,-,N} \right) 
&= \mu_{r,+,N}, \label{eq:local-quotient-system-plus} \\ 
u_{r,-,N} + \rho_N \frac{\beta_{r,N}}{n_{r,-,N}} \left( u_{r,-,N}-u_{r,+,N} \right) 
&= \mu_{r,-,N}.
\label{eq:local-quotient-system-minus}
\end{align}

Subtracting Eq.~\eqref{eq:local-quotient-system-minus} from Eq.~\eqref{eq:local-quotient-system-plus} gives
\begin{align}
\left[ 1+ \rho_N\beta_{r,N} \left( n_{r,+,N}^{-1} + n_{r,-,N}^{-1} \right) \right] \left( u_{r,+,N}-u_{r,-,N} \right) = \mu_{r,+,N} - \mu_{r,-,N}.
\end{align}
Using the definitions of \(c_{r,N}\) and \(\Delta_{r,N}^\star\), this becomes
\begin{equation}
(1+\rho_Nc_{r,N}) \left( u_{r,+,N}-u_{r,-,N} \right) = \Delta_{r,N}^\star.
\label{eq:local-contrast-equation}
\end{equation}
Hence
\begin{equation}
u_{r,+,N}-u_{r,-,N} = \frac{1}{1+\rho_Nc_{r,N}} \Delta_{r,N}^\star.
\label{eq:local-quotient-contrast}
\end{equation}

Since $M_{r,N}J_{r,N}=I_2$, Eq.~\eqref{eq:expected-estimator-quotient-form} implies
\begin{align}
\Delta_{r,N} \left( \mathbb E[ \widehat V_{r,N,\rho_N} ] \right) 
&= d^\top M_{r,N} J_{r,N} U_{r,N,\rho_N} \nonumber\\ 
&= d^\top U_{r,N,\rho_N} \nonumber\\ 
&= u_{r,+,N} - u_{r,-,N}.
\end{align}
Therefore, by Eq.~\eqref{eq:local-quotient-contrast}, we have
\begin{equation}
\Delta_{r,N}
\left(
\mathbb E[
\widehat V_{r,N,\rho_N}
]
\right)
=
\frac{1}{1+\rho_Nc_{r,N}}
\Delta_{r,N}^\star.
\label{eq:local-contrast-attenuation}
\end{equation}

Subtracting the true contrast yields
\begin{align}
\Delta_{r,N} \left( \mathbb E[ \widehat V_{r,N,\rho_N} ] \right) - \Delta_{r,N}^\star = - \frac{\rho_Nc_{r,N}} {1+\rho_Nc_{r,N}} \Delta_{r,N}^\star.
\end{align}
Consequently, we have
\begin{equation}
\left\| \Delta_{r,N} \left( \mathbb E[ \widehat V_{r,N,\rho_N} ] \right) - \Delta_{r,N}^\star \right\|_2^2 = \left( \frac{\rho_Nc_{r,N}} {1+\rho_Nc_{r,N}} \right)^2 \|\Delta_{r,N}^\star\|_2^2.
\label{eq:local-contrast-bias-exact}
\end{equation}

Because \(\Delta_{r,N}\) is linear, we have
\[
\mathbb E \left[ \Delta_{r,N} \left( \widehat V_{r,N,\rho_N} \right) \right] = \Delta_{r,N} \left( \mathbb E[ \widehat V_{r,N,\rho_N} ] \right).
\]
The bias--variance identity therefore gives
\begin{align}
& \phantom{=} \mathbb E \left[ \left\| \Delta_{r,N} \left( \widehat V_{r,N,\rho_N} \right) - \Delta_{r,N}^\star \right\|_2^2 \right] \nonumber\\ 
& = \left\| \Delta_{r,N} \left( \mathbb E[ \widehat V_{r,N,\rho_N} ] \right) - \Delta_{r,N}^\star \right\|_2^2 + \mathbb E \left[ \left\| \Delta_{r,N} \left( \widehat V_{r,N,\rho_N} \right) - \Delta_{r,N} \left( \mathbb E[ \widehat V_{r,N,\rho_N} ] \right) \right\|_2^2 \right] \nonumber\\ 
& \geq \left\| \Delta_{r,N} \left( \mathbb E[ \widehat V_{r,N,\rho_N} ] \right) - \Delta_{r,N}^\star \right\|_2^2.
\label{eq:local-bias-variance-lower-bound}
\end{align}
Equivalently, the same inequality follows from Jensen's inequality.

Using Eq.~\eqref{eq:local-contrast-bias-exact}, we obtain
\begin{align}
\mathbb E \left[ \left\| \Delta_{r,N} \left( \widehat V_{r,N,\rho_N} \right) - \Delta_{r,N}^\star \right\|_2^2 \right] \geq \left( \frac{\rho_Nc_{r,N}} {1+\rho_Nc_{r,N}} \right)^2 \|\Delta_{r,N}^\star\|_2^2.
\label{eq:local-contrast-risk-lower-bound}
\end{align}
Multiplying by \(\omega_{r,N}\) and summing over \(r=1,\ldots,m_N\) yields
\begin{align}
\mathbb E \left[ \mathcal L_{\mathrm{contrast},N} \left( \widehat V_{\rho_N}, V_N^\star \right) \right] \geq \sum_{r=1}^{m_N} \omega_{r,N} \left( \frac{\rho_Nc_{r,N}} {1+\rho_Nc_{r,N}} \right)^2 \|\Delta_{r,N}^\star\|_2^2.
\label{eq:global-bias-sum}
\end{align}
Restricting the sum to \(r\in\mathcal R_N\) yields
\begin{align}
\mathbb E \left[ \mathcal L_{\mathrm{contrast},N} \left( \widehat V_{\rho_N}, V_N^\star \right) \right] \geq \sum_{r\in\mathcal R_N} \omega_{r,N} \left( \frac{\rho_Nc_{r,N}} {1+\rho_Nc_{r,N}} \right)^2 \|\Delta_{r,N}^\star\|_2^2.
\label{eq:restricted-global-bias-sum}
\end{align}

The function
\[
g(x) = \frac{x}{1+x}, \qquad x\geq0,
\]
is nondecreasing. Hence, for every \(r\in\mathcal R_N\), inequality $\rho_Nc_{r,N}\geq c_0$ implies
\[
\frac{\rho_Nc_{r,N}} {1+\rho_Nc_{r,N}} \geq \frac{c_0}{1+c_0}.
\]
Together with $\|\Delta_{r,N}^\star\|_2 \geq \delta_v$, we obtain
\begin{align}
\mathbb E \left[ \mathcal L_{\mathrm{contrast},N} \left( \widehat V_{\rho_N}, V_N^\star \right) \right] 
&\geq \sum_{r\in\mathcal R_N} \omega_{r,N} \left( \frac{c_0}{1+c_0} \right)^2 \delta_v^2 \nonumber\\ 
& = \left( \sum_{r\in\mathcal R_N} \omega_{r,N} \right) \left( \frac{c_0}{1+c_0} \right)^2 \delta_v^2 \nonumber\\ 
&\geq \omega_0 \left( \frac{c_0}{1+c_0} \right)^2 \delta_v^2.
\end{align}
This proves Eq.~\eqref{eq:finite-n-contrast-lower-bound}. Taking the lower limit as \(N\to\infty\) gives Eq.~\eqref{eq:asymptotic-contrast-lower-bound}.
\end{proof}

\end{appendices}

\bibliography{sn-bibliography}
\end{document}